\documentclass[12pt]{amsart}

\pdfoutput=1

\usepackage[utf8]{inputenc}
\usepackage[margin=.8in]{geometry}
\usepackage{amsmath,amsthm,amsfonts,amssymb,amscd,cite,graphicx}
\usepackage{url}
\usepackage{stmaryrd}
\usepackage{amstext}
\usepackage{verbatim}
\usepackage{caption}
\usepackage[labelformat=simple,labelfont={}]{subcaption}

\theoremstyle{definition}
\newtheorem{theorem}{Theorem}[section]

\newtheorem{corollary}[theorem]{Corollary}

\newtheorem{example}[theorem]{Example}

\newtheorem{question}[theorem]{Question}

\newtheorem{proposition}[theorem]{Proposition}
\newtheorem{lemma}[theorem]{Lemma}

\allowdisplaybreaks[4]

\usepackage{babel,tikz}
\usetikzlibrary{arrows,shapes,positioning,backgrounds,calc}
\usetikzlibrary{arrows}
\usetikzlibrary{shapes} 
\usetikzlibrary{patterns}
\tikzstyle{vert} = 
[circle, draw, fill=grey!40,inner sep=0pt, minimum size=5mm]
\tikzstyle{small vert} = 
[circle, draw,fill=grey!40, inner sep=0pt, minimum size=3.75mm]
\tikzstyle{tiny vert} = 
[circle, draw,fill=grey!40, inner sep=1.6pt, minimum size=1mm]
\tikzstyle{rect vert} = 
[rectangle, rounded corners, draw, fill=grey!40,inner sep=2.7pt, minimum size=4mm]
\tikzstyle{b} = [draw, very thick, black,-]
\tikzstyle{d} = [draw, thick, black,-stealth]
\tikzstyle{a} = [draw, very thick, black,-stealth]

\definecolor{grey}{rgb}{.7, .7, .7}

\definecolor{orange}{RGB}{255,102,0}
\definecolor{ggreen}{RGB}{0,153,0}
\definecolor{darkblue}{RGB}{0,0,255}
\definecolor{purple}{RGB}{153,51,255}
\definecolor{turq}{RGB}{72,209,204}
\definecolor{gray}{RGB}{220,220,220}
\definecolor{orange2}{RGB}{255,100,0}
\definecolor{purple2}{RGB}{159,51,250}
\definecolor{rred}{rgb}{0.9, 0.17, 0.31}
\definecolor{naugreen}{cmyk}{.43,0,.34,.38}
\definecolor{naublue}{cmyk}{1,.72,0,.32}
\definecolor{mediterranean}{cmyk}{.67,0,.08,.3}
\definecolor{rose}{cmyk}{0,1.00,.20,0}
\definecolor{darkorchid}{cmyk}{.6,.9,0,.05}
\definecolor{butterfly}{cmyk}{.95,.59,0,.10}
\definecolor{springgreen}{cmyk}{1.00,0,.70,.02}
\definecolor{darkred}{cmyk}{0,1,1,.5}
\definecolor{nectarine}{cmyk}{0,0.70,1.00,0}
\definecolor{icyblue}{cmyk}{.84,.25,0,.06}
\definecolor{manatee}{rgb}{0.59, 0.6, 0.67}

\newcommand{\convexpath}[2]{
[
create hullnodes/.code={
\global\edef\namelist{#1}
\foreach [count=\counter] \nodename in \namelist {
\global\edef\numberofnodes{\counter}
\node at (\nodename) [draw=none,name=hullnode\counter] {};
}
\node at (hullnode\numberofnodes) [name=hullnode0,draw=none] {};
\pgfmathtruncatemacro\lastnumber{\numberofnodes+1}
\node at (hullnode1) [name=hullnode\lastnumber,draw=none] {};
},
create hullnodes
]
($(hullnode1)!#2!-90:(hullnode0)$)
\foreach [
evaluate=\currentnode as \previousnode using \currentnode-1,
evaluate=\currentnode as \nextnode using \currentnode+1
] \currentnode in {1,...,\numberofnodes} {
  let
\p1 = ($(hullnode\currentnode)!#2!-90:(hullnode\previousnode)$),
\p2 = ($(hullnode\currentnode)!#2!90:(hullnode\nextnode)$),
\p3 = ($(\p1) - (hullnode\currentnode)$),
\n1 = {atan2(\y3,\x3)},
\p4 = ($(\p2) - (hullnode\currentnode)$),
\n2 = {atan2(\y4,\x4)},
\n{delta} = {-Mod(\n1-\n2,360)}
  in 
{-- (\p1) arc[start angle=\n1, delta angle=\n{delta}, radius=#2] -- (\p2)}
}
-- cycle
}

\DeclareMathOperator{\Opt}{Opt}
\DeclareMathOperator{\Con}{Con}

\begin{document}

\global\long\def\RR{\sf{R}}%
\global\long\def\SS {\sf{S}}%
\global\long\def\v#1{\left[#1\right]}%
\global\long\def\bb#1{\left\llbracket#1\right \rrbracket}%
\global\long\def\set#1{\left\{#1\right \}}%
\global\long\def\l#1{\left(#1\right)}%
\global\long\def\tilda#1{\widetilde{#1}}%
\global\long\def\opt#1{\text{Opt}\left(#1\right)} % N: not a fan of this makro
\global\long\def\im#1{\text{im}\left(#1\right)}%
\global\long\def\G{\mathbb{G}}%
\global\long\def\H{\mathbb{H}}%

\newcommand{\note}[1]{{\color{olive}\small#1}}
\newcommand{\newold}[2]{{\bf new:}{\color{blue}#1}{\bf old:}{\color{red}#2}}%
\newcommand{\CC}{\mathsf{C}}
\newcommand{\DD}{\mathsf{D}}

\tolerance=1000

\title{Isomorphism Theorems for Impartial Combinatorial Games}

\author{Mikhail Baltushkin}
\author{Dana C.~Ernst}
\author{N\'andor Sieben}

\email{Mikhail.Baltushkin@nau.edu,Dana.Ernst@nau.edu,Nandor.Sieben@nau.edu}

\address{Northern Arizona University, Department of Mathematics and Statistics, Flagstaff, AZ 86011-5717, USA}
	
\begin{abstract}
We introduce the category of optiongraphs and option-preserving maps as a model to study impartial combinatorial games. Outcomes, remoteness, and extended nim-values are preserved under option-preserving maps. We show that the four isomorphism theorems from universal algebra are valid in this category. Quotient optiongraphs, including the minimum quotient, provide simplifications that can help in the analysis of games.
\end{abstract}
	
\keywords{impartial game; option-preserving map; congruence relation; minimum quotient; isomorphism theorems}
\subjclass[2020]{91A46, 91A43, 05C57, 08A30}

\maketitle
	
\section{Introduction}
	
	Impartial combinatorial games are commonly modeled using a digraph. The vertices are the positions and the arrows describe how to move between these positions. The out-neighbors of a position are the possible options or followers of the position.  During a play, two players take turns replacing the current position with one of its options by moving along the arrows. That is, a play is a directed walk in the digraph from a designated starting position. A play ends at positions with no options, which we refer to as terminal positions. The outcome of the play is determined at these terminal positions. There are two common play conventions. Under normal play, the last player to move is the winner. Under mis\`ere play, the last player to move is the loser. Infinite play may occur if the digraph has an infinite directed walk, possibly due to the presence of cycles or infinite paths. The outcome of any infinite play is declared a draw.  Fixing a specific starting position creates a game. There are three possible outcomes in such a game: either the first or the second player can force a win in each play, or both players can force a draw in each play. This determines an outcome function on the set of positions.
	
	The endgame of many combinatorial games decomposes as a sum of games, so game sums play a very important role in the theory. The most important tool for analysis of game sums under normal play is the nim-value, also referred to as Grundy-value. The nim-value determines the outcome of a game, and it allows for the easy computation of the nim-value of the sum. The original theory of nim-values  developed for well-founded digraphs without infinite play in~\cite{Sprague,Grundy} was extended by Smith to the theory of (extended) nim-values on finite digraphs in~\cite{SmithLoopy}. Smith's paper envisions that the development works for infinite digraphs. A fully general development for infinite digraphs can be found in~\cite{Li}.  Also see~\cite{siegel2005coping,SiegelThesis,SiegelBook} for results about games with cycles, under the name loopy games. The extended nim theory uses the remoteness values associated to the positions. Essentially, the remoteness measures how quickly the winner can win and how long the loser can delay losing. 
	
	In this paper, we introduce a category where the objects are digraphs and the morphisms are so-called option-preserving maps. In this category, we refer to each object as an optiongraph. Each option-preserving map preserves essential information related to the positions, including outcome, remoteness, and nim-value~\cite{Li}.   An optiongraph without any infinite plays is called a rulegraph. In~\cite{Basic2024}, the authors study the category of rulegraphs and option-preserving maps.  Their model enables the construction of quotient rulegraphs compatible with nim-value and formal birthday.  The smallest such quotient aligns with Conway's description of an impartial game in~\cite[Chapter 11]{ONAG}. Intermediate quotients allow us to discard irrelevant details while retaining essential information. Finding the right balance preserves intuitive understanding and facilitates game analysis. The authors also include analogs for the First and Fourth Isomorphism Theorems in universal algebra.
	
	One goal of this paper is to provide a categorical framework that encompasses all optiongraphs, and contains morphisms that preserve all essential information of positions. Our main result is that this category supports the four well-known isomorphism theorems from universal algebra.
	
	Other authors have studied maps on digraphs in the context of impartial combinatorial games that are similar to option-preserving maps.  For example,~\cite{BanerjiBook,BanerjiErnst,FraenkelPerl,FraenkelYesha,FraenkelYeshaAnnihilation} are concerned with $D$-morphisms, which are maps that allow for more identification than option-preserving maps by permitting certain arrow reversals. For rulegraphs, $D$-morphisms preserve nim-value, and hence outcome, but this is not true for cyclic optiongraphs. This indicates that $D$-morphisms are not the appropriate morphisms for our purposes.

	For a comprehensive treatment of the standard theory of impartial games, consult~\cite{albert2007lessons,ONAG,SiegelBook}. 
	
\section{Preliminaries}\label{sec:prelims}
		
	A \emph{digraph} $D$ is a pair $(V,E)$ where $V$ is a nonempty set of \emph{vertices} and $E\subseteq V\times V$ is the set of \emph{arrows}. We allow $V$ to be infinite and $E$ to contain loops. A \emph{digraph homomorphism} is a function between the vertex sets of two digraphs that maps arrows to arrows. Digraphs and digraph homomorphisms form the category {\bf Gph}. 
	
	We use digraphs to model impartial combinatorial games, where we think of the vertices as positions and arrows as possible moves between positions. We call the elements of the set of out-neighbors of a position $p$ the \emph{options} of $p$. An \emph{optiongraph} is a nonempty set $\DD$ of positions together with an option function $\Opt_{\DD}:\DD\to 2^\DD$. The option function $\Opt_{\DD}$, or simply $\Opt$, encodes the same information as the arrow set, so an optiongraph is essentially a digraph.
	
	During a \emph{play} on an optiongraph, two players take turns replacing the current position with one of its options. So a play is essentially a walk in the optiongraph from a chosen starting position. A position $q$ is a \emph{subposition} of position $p$ if there is a finite walk from $p$ to $q$. A play ends when the current position becomes a \emph{terminal position} without options. In \emph{normal play}, the last player to move wins. In \emph{mis\`ere play}, the last player to move loses. Infinite play is a possibility, in which case the play is considered a draw.
	
	A function $f:\CC\to\DD$ is called \emph{option preserving} if $\Opt_\DD(f(p))=f(\Opt_\CC(p))$ for all $p\in\CC$. This is a generalization of~\cite[Definition 4.1]{Basic2024} for rulegraphs. Option-preserving maps are confusingly called homomorphisms in~\cite[Section 4]{Li}. Example~\ref{ex:FirstIso} shows an option-preserving map.
	The composition of option-preserving functions is clearly option preserving. Optiongraphs and option-preserving maps form the category {\bf OGph}. It is easy to see that an option-preserving map is a digraph homomorphism. The converse is false. So {\bf OGph} is a wide but not full subcategory of {\bf Gph}. This provides the motivation for the renaming of digraphs as optiongraphs in this context.
	If an optiongraph has no infinite play, then we call it a \emph{rulegraph}~\cite{Basic2024}. 
	
	Option-preserving maps preserve so-called valuations on rulegraphs~\cite[Proposition~6.10]{Basic2024}, that is, functions on the positions defined recursively via Opt. In particular, they preserve the formal birthday, nim-value, and outcome of the positions both in normal and mis\`ere play. An option-preserving map also preserves the remoteness and the extended nim-value of an optiongraph~\cite[Theorem 6]{Li}. 
	It seems plausible that an option-preserving map that is compatible with the designation of entailing positions, also preserves the affine nim-values used in~\cite{entailing}.
	
		An optiongraph $\DD$ is an $F$-coalgebra  where $F:{\bf Set}\to{\bf Set}$ is the power set functor on the category of sets and $\Opt:\DD\to F(\DD)$ is a morphism of {\bf Set}.
		Our option-preserving maps are $F$-coalgebra homomorphisms. Since our focus is applications in combinatorial game theory, we develop our results from first principles and do not rely on this viewpoint to keep our results as accessible as possible.
	
\section{Partitions of Optiongraphs}
	
\subsection{Finite-Infinite-Mixed partition}
	An optiongraph $\mathsf{D}$ can be partitioned into three sets $F_{\mathsf{D}}$, $I_{\mathsf{D}}$, and $M_{\mathsf{D}}$, or simply $F$, $I$, $M$ if it is clear from context. Positions without a terminal subposition belong to $I$. Positions that are not the starting vertex of any infinite play belong to $F$. The rest of the positions belong to $M$. It is clear that $\bigcup\Opt(F)\subseteq F$ and $\bigcup\Opt(I)\subseteq I$. The equivalence classes $I$ and $F$ are well understood. Class $I$ contains draw positions, while $F$ is the set of positions of a rulegraph. Class $M$ contains positions that are the most difficult to handle. 
	\begin{proposition}
		If $f:\CC\to\DD$ is an option-preserving surjection, then $f(F_\CC)=F_\DD$, $f(I_\CC)=I_\DD$, and $f(M_\CC)=M_\DD$.
	\end{proposition}
	
	\begin{proof}
		Using \cite[Proposition 4.20]{Basic2024}, it is easy to see that the image of a play is a play, and that a play in $\DD$ is an image of a play in $\CC$. Moreover, the length of a play (finite or infinite) is preserved in both directions. 
	\end{proof}

The previous result is useful to speed up the computation of finding option-preserving maps.

	\subsection{Quotient optiongraphs}\label{sec:quotients}
	We write $[p]_\theta$ or simply $[p]$ to denote the equivalence class of $p$ with respect to the equivalence relation $\theta$. An equivalence relation $\theta$ on  an optiongraph $\mathsf{D}$ is a \emph{congruence relation} if $p\mathrel{\theta} q$ implies $\v{\Opt(p)}=\v{\Opt(q)}$, where $\v S:=\{\v s\mid s\in S\}$. Congruence relations for rulegraphs were defined in~\cite{Basic2024}.
	The kernel $\ker(f):=\{(p,q)\in \CC\times\CC \mid f(p)=f(q)\}$ of an option-preserving map $f:\CC\to\DD$ is an example of a congruence relation by Theorem~\ref{thm:first iso}.

	For a congruence relation $\theta$ on $\mathsf{D}$, the \emph{quotient optiongraph} $\mathsf{D}/\theta$ has option function $\Opt_{\theta}([p]):=\Opt_{\DD/\theta}([p]):=[\Opt(p)]$ for all $p\in \DD$.
	This is well defined since $\theta$ is a congruence relation. The \emph{canonical quotient map} $f:\mathsf{D}\to \mathsf{D}/{\theta}$ defined by $f(p):=[p]$ is option preserving since 
	\[
	\Opt_{\theta}(f(p))=\Opt_{\theta}([p])=[\Opt(p)]=\{[q]\mid q\in\Opt(p)\}=\{f(q)\mid q\in\Opt(p)\}=f(\Opt(p)).
	\]
	Note that the notation $\Opt([p])$ is ambiguous. In this case we always mean $\Opt_\theta([p])$ and not $\Opt_\DD([p])$.
	
	Let $\Con(\mathsf{D})$ be the set of congruence relations on an optiongraph $\mathsf{D}$. For every optiongraph $\mathsf{D}$ there is a \emph{maximum congruence relation} $\bowtie_{\mathsf{D}}$, or simply $\bowtie$, on $\mathsf{D}$. We call $\mathsf{D}/{\bowtie}$ the \emph{minimum quotient} of $\mathsf{D}$. The existence of $\bowtie$ appears in~\cite[Theorem 7]{Li}. For rulegraphs this was shown in~\cite[Proposition 7.5]{Basic2024}. Our terminology and notation is quite different from that of~\cite{Li}, so we include the short proof following the development suggested in~\cite[Remark~7.6]{Basic2024} for rulegraphs.
	
	\begin{lemma}\label{lem:transitiveClosure}
		The transitive closure $\theta$ of the union of $\phi,\psi\in \Con(\DD)$ for an optiongraph $\DD$ is also a congruence relation.
	\end{lemma}
	
	\begin{proof}
		Assume $(p,q)\in\theta$. Then 
		\[
		p=r_0\mathrel{\eta_0} r_1 \mathrel{\eta_1} \cdots \mathrel{\eta_{k-2}} r_{k-1} \mathrel{\eta_{k-1}} r_k=q
		\]
		for some positions $r_0,\ldots,r_k$ and $\eta_0,\ldots,\eta_{k-1}\in\{\phi,\psi\}$. Hence $[\Opt(r_i)]_{\eta_i}=[\Opt(r_{i+1})]_{\eta_i}$ for all $i$. Since $[s]_{\eta_i}\subseteq[s]_{\theta}$ for all $s\in\Opt(r_i)$, this implies that $[\Opt(r_i)]_{\theta}=[\Opt(r_{i+1})]_{\theta}$ for all $i$.
	\end{proof}
	
	\begin{proposition}\label{prop:bowtie}
		The union ${\bowtie}:=\bigcup\Con(\DD)$ is a congruence relation on the optiongraph $\mathsf{D}$. 
	\end{proposition}
	
	\begin{proof}
		The relation $\bowtie$ is clearly reflexive and symmetric. If $(p,q),(q,r)\in{\bowtie}$, then $(p,q)\in\phi$ and $(q,r)\in\psi$ for some $\phi,\psi\in \Con(\DD)$. Hence $(p,r)\in\theta\subseteq{\bowtie}$ by Lemma~\ref{lem:transitiveClosure}, where $\theta$ is the transitive closure of $\phi$ and $\psi$. This implies that $\bowtie$ is also transitive. 
		Assume $(p,q)\in{\bowtie}$, so that $(p,q)\in\theta$ for some $\theta\in\Con(\DD)$. Then $[\Opt(p)]_{\theta}=[\Opt(q)]_{\theta}$, and hence $[\Opt(p)]_{\bowtie}=[\Opt(q)]_{\bowtie}$ since $\theta\subseteq{\bowtie}$.
	\end{proof}
	
	In~\cite{Basic2024}, the authors prove that $\Con(\mathsf{R})$ is a complete lattice when $\mathsf{R}$ is a rulegraph. We extend this to optiongraphs.
	
	\begin{proposition}\label{prop:complete lattice}
		The set $\Con(\DD)$ of congruence relations on an optiongraph $\DD$ forms a complete lattice under inclusion.
	\end{proposition}
	
	\begin{proof}
		The proof presented in~\cite[Proposition 7.16]{Basic2024} carries over to optiongraphs. The meet and join of $\Theta\subseteq \Con(\DD)$ are defined by $\bigwedge \Theta:=\bigcap \Theta$ and $\bigvee \Theta:=\bigwedge\{\theta\in \Con(\DD)\mid \bigcup \Theta\subseteq \theta\}$.
	\end{proof}
	
	\begin{example} \rm
		Figure~\ref{fig:Con} shows an optiongraph and its lattice of congruence relations. Congruence relations are indicated by a list of nontrivial congruence classes separated by bars. 
		For example, the classes of ${\bowtie}=ab|st|wz$ are $\{a,b\}$, $\{r\}$, $\{s,t\}$, and $\{w,z\}$.
	\end{example}
	
	\begin{figure}[h]
		\hfil
		\begin{tikzpicture}[xscale=1]
			\node (w) [small vert,fill=yellow!70] at (0,.5) {$\scriptstyle w$};
			\node (a) [small vert] at (2,1) {$\scriptstyle a$};
			\node (b) [small vert] at (2,0) {$\scriptstyle b$};
			\node (s) [small vert,fill=green!40] at (3,1) {$\scriptstyle  s$};
			\node (t) [small vert,fill=green!40] at (3,0) {$\scriptstyle  t$};
			\node (z) [small vert,fill=yellow!70] at (1,.5) {$\scriptstyle  z$};
			\node (r) [small vert,fill=lime!45] at (3.8,.5) {$\scriptstyle  r$};
			\path [d] (w) to (z);
			\path [d] (a) to (b);
			\path [d] (b) to (a);
			\path [d] (a) to (s);
			\path [d] (b) to (t);
			\path [d] (b) to (z);
			\path [d] (a) to (z);
			\path [d] (r) to (s);
			\path [d] (r) to (t);
			\path [d,loop left,distance=7mm] (w) to (w);
			\path [d,loop above,distance=7mm] (z) to (z);
			\path [d,loop above,distance=7mm] (a) to (a);
			\node  at (2,-.5) {$\mathsf{D}$};
		\end{tikzpicture}
		\hfil %%%%%%%%%%%%%%%%%%%%%%%%%%%%%%%%%%%%%%%%%%%%%%%%%%%%%%%%%%%%%%
		\begin{tikzpicture}[xscale=1,yscale=.6]
			\node (abcxy) [rect vert] at (-1,3) {$\scriptstyle ab|st|wz$};
			\node (ay cx) [rect vert] at (-2,2) {$\scriptstyle ab|st$};
			\node (abcy) [rect vert] at (0,2) {$\scriptstyle st|wz$};
			\node (abc) [rect vert] at (0,1) {$\scriptstyle wz $};
			\node (cx) [rect vert] at (-2,1) {$\scriptstyle st $};
			\node (none) [rect vert] at (-1,0.4) {$\scriptstyle  $};
			\path [d,-] (abcxy) to (ay cx);
			\path [d,-] (abcxy) to (abcy);
			\path [d,-] (abcy) to (abc);
			\path [d,-] (cx) to (none);
			\path [d,-] (abc) to (none);
			\path [d,-] (ay cx) to (cx);
			\path [d,-] (abcy) to (cx);
			\node  at (-1,-.5) {$\Con(\mathsf{D})$};
		\end{tikzpicture}
		\caption{
			\label{fig:Con}
			Lattice of congruence relations of an optiongraph.
		}
	\end{figure}
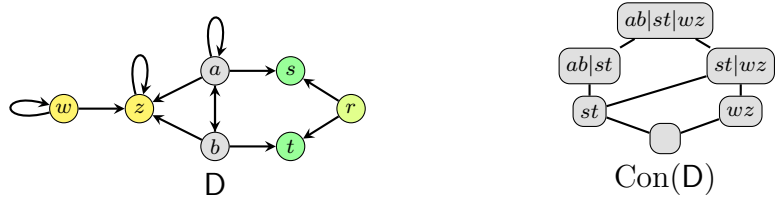
	
	\begin{example}\label{ex:extendednims} \rm
		Figure~\ref{fig:loopyEx} shows an optiongraph $\mathsf{D}$ and its minimum quotient $\mathsf{D}/{\bowtie}$. The figure also shows the extended nim-values and the remoteness values~\cite{Li} of the positions of $\mathsf{D}/{\bowtie}$, which can be lifted to $\DD$. Note that $\DD$ contains both a cycle and an infinite path. The unlabeled positions of $\mathsf{D}$ belong to $I$, while $F=\{s,t,c\}$ and $M=\{a,b\}$. 
	\end{example}
	
	\begin{figure}[h]
		\hfil
		\begin{tikzpicture}[yscale=.8,xscale=.8]
			\node (a) [small vert] at (1,0) {$\scriptstyle a$};
			\node (t) [small vert,fill=green!40] at (0,0) {$\scriptstyle t$};
			\node (b) [small vert] at (1,1) {$\scriptstyle b$};
			\node (y) [small vert,fill=yellow!70] at (2,1) {$\scriptstyle $};
			\node (x) [small vert,fill=yellow!70] at (3,1) {$\scriptstyle $};
			\node (c) [small vert,fill=lime!45] at (0,1) {$\scriptstyle c$};
			\node (s) [small vert,fill=green!40] at (-1,1) {$\scriptstyle s$};
			\node (z) [small vert,fill=yellow!70] at (2,0) {$\scriptstyle $};
			\node (z1) at (3,0) {$\cdots $};
			
			\path [d] (a) to (b);
			\path [d] (a) to (t);
			\path [d] (b) to (y);
			\path [d] (b) to (c);
			\path [d] (c) to (s);
			\path [d] (b) to (z);
			\path [d] (y) to (x);
			\path [d] (z) to (z1);
			\path [d] (a) to (c);
			\path [d,loop right,distance=7mm] (x) to (x);
			
			\node  at (1.5,-.8) {$\mathsf{D}$};   
		\end{tikzpicture}
		\hfil %%%%%%%%%%%%%%%%%%%%%%%%%%%%%%%%%%%%%%
		\begin{tikzpicture}[yscale=.8,xscale=1]
			\node (a) [rect vert] at (1,0) {$\scriptstyle \{a\}$};
			\node (st) [rect vert,fill=green!40] at (0,0) {$\scriptstyle \{s,t\}$};
			\node (b) [rect vert] at (1,1) {$\scriptstyle \{b\}$};
			\node (c) [rect vert,fill=lime!45] at (0,1) {$\scriptstyle \{c\}$};
			\node (xyz) [rect vert,fill=yellow!70] at (2,1) {$\scriptstyle\ $};
			
			\path [d] (a) to (b);
			\path [d] (a) to (st);
			\path [d] (b) to (xyz);
			\path [d] (b) to (c);
			\path [d] (c) to (st);
			\path [d] (a) to (c);
			\path [d,loop right,distance=7mm] (xyz) to (xyz);
			
			\node  at (1,-.8) {$\mathsf{D}/{\bowtie}$};   
		\end{tikzpicture}
		\hfil %%%%%%%%%%%%%%%%%%%%%%%%%%%%%%%%%%%%%%
		\begin{tikzpicture}[yscale=.8,xscale=1]
			\node (a) [rect vert] at (1,0) {$\scriptstyle \infty_{0,1}$};
			\node (st) [rect vert,fill=green!40] at (0,0) {$\scriptstyle 0$};
			\node (b) [rect vert] at (1,1) {$\scriptstyle \infty_1$};
			\node (c) [rect vert,fill=lime!45] at (0,1) {$\scriptstyle 1$};
			\node (xyz) [rect vert,fill=yellow!70] at (2,1) {$\scriptstyle\infty $};
			
			\path [d] (a) to (b);
			\path [d] (a) to (st);
			\path [d] (b) to (xyz);
			\path [d] (b) to (c);
			\path [d] (c) to (st);
			\path [d] (a) to (c);
			\path [d,loop right,distance=7mm] (xyz) to (xyz);
			
			\node  at (1.3,-.8) {$\mathsf{D}/{\bowtie}$ nim-values};   
		\end{tikzpicture}
		\hfil %%%%%%%%%%%%%%%%%%%%%%%%%%%%%%%%%%%%%%
		\begin{tikzpicture}[yscale=.8,xscale=1]
			\node (a) [rect vert] at (1,0) {$\scriptstyle 1$};
			\node (st) [rect vert,fill=green!40] at (0,0) {$\scriptstyle 0$};
			\node (b) [rect vert] at (1,1) {$\scriptstyle \infty$};
			\node (c) [rect vert,fill=lime!45] at (0,1) {$\scriptstyle 1$};
			\node (xyz) [rect vert,fill=yellow!70] at (2,1) {$\scriptstyle\infty $};
			
			\path [d] (a) to (b);
			\path [d] (a) to (st);
			\path [d] (b) to (xyz);
			\path [d] (b) to (c);
			\path [d] (c) to (st);
			\path [d] (a) to (c);
			\path [d,loop right,distance=7mm] (xyz) to (xyz);
			
			\node  at (1.3,-.8) {$\mathsf{D}/{\bowtie}$ remoteness};   
		\end{tikzpicture}
		
		\caption{\label{fig:loopyEx}
			An optiongraph $\DD$, its minimum quotient $\DD/{\bowtie}$, together with extended nim-values and remoteness on $\DD/{\bowtie}$. 
		}
	\end{figure}
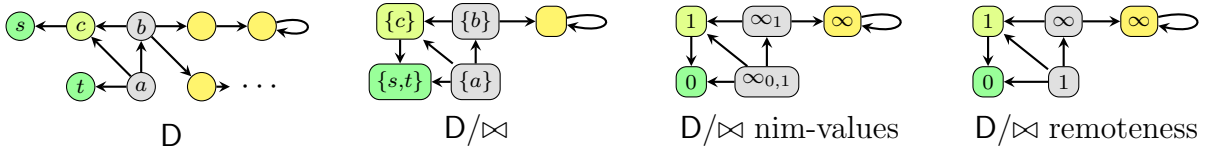
	
	Our companion web page~\cite{optiongraphs_companion} provides code for finding the congruence relations of a finite optiongraph. The code uses the CPMpy constraint programming Python library~\cite{guns2019increasing}.
	
	\begin{question}
		\rm
		What lattices can arise as the lattice of congruence relations of an optiongraph? 
	\end{question}
	
	\section{First Isomorphism Theorem}\label{sec:first iso thm}
	
	An optiongraph $\mathsf{C}$ is a \emph{suboptiongraph} of an optiongraph $\mathsf{D}$ if $\mathsf{C}\subseteq \mathsf{D}$ and the inclusion map $\mathsf{C}\hookrightarrow \mathsf{D}$ is option preserving. We use the notation $\mathsf{C}\le\mathsf{D}$ to indicate this relationship.

	\begin{example} \rm
		When nonempty, both $F_{\mathsf{D}}$ and $I_{\mathsf{D}}$ form the positions of suboptiongraphs $\mathsf{F}_{\mathsf{D}}$ and $\mathsf{I}_{\mathsf{D}}$ of $\mathsf{D}$, respectively. Suboptiongraph $\mathsf{F}_{\mathsf{D}}$ is actually a rulegraph.
	\end{example}
	
	If $f:\mathsf{C}\to \mathsf{D}$ is option preserving, then $f(\CC)$ is closed under the $\Opt_\mathsf{D}$ function. So $f(\CC)$ is a suboptiongraph of $\DD$. 
	
	\begin{theorem}[First Isomorphism]\label{thm:first iso}\rm 
		If $f:\mathsf{C}\to\mathsf{D}$ is an option-preserving map, then $\ker(f)\in\Con(\CC)$ and $f(\mathsf{C})$ is isomorphic to the quotient $\mathsf{Q}:=\mathsf{C}/\ker(f)$.
	\end{theorem}
	
	\begin{proof}
		Under $\ker(f)$ we have
		\begin{align*}
		[\Opt_\mathsf{C}(p)]
		&=\{\{r\mid f(r)=f(q)\}\mid q\in\Opt_\mathsf{C}(p)\}\\
		&=\{\{r\mid f(r)=s\}\mid s\in\Opt_\mathsf{D}(f(p))\}\\
		&=\{f^{-1}(\{s\})\mid s\in\Opt_\mathsf{D}(f(p)) \}.
		\end{align*}
		This formula shows that $f(p)=f(q)$ implies $[\Opt_\mathsf{C}(p)]=[\Opt_\mathsf{C}(q)]$. So $\ker(f)$ is a congruence relation on $\mathsf{C}$. 
		Since $f(q)\in \Opt_\mathsf{D}(f(p))=f(\Opt_\mathsf{C}(p))$ if and only if $[q]\in\Opt_\mathsf{Q}([p])=[\Opt_\mathsf{C}(p)]$, the well-defined map $[p]\mapsto f(p):\mathsf{Q}\to f(\mathsf{C})$ is an isomorphism.
	\end{proof}
	
	\begin{example} \label{ex:FirstIso} \rm
		Figure~\ref{fig:1stIso} demonstrates the First Isomorphism Theorem, where the option-preserving map $f:\mathsf{C}\to\mathsf{D}$ takes $a,b,c$ to $y$ and $d$ to $z$. Note that $(a,b)\in\ker(f)$ but $\Opt(a)=\{b,d\}\ne\{c,d\}=\Opt(b)$, which illustrates that congruent positions need not have the same option sets. Note that ${F}_{\mathsf{C}}=\{d\}$, $I_{\mathsf{C}}=\emptyset$, $M_{\mathsf{C}}=\{a,b,c\}$, and $\ker(f)={\bowtie_\CC}$.
	\end{example}
	
	\begin{figure}[h]
		\hfil
		\begin{tikzpicture}[yscale=.9,xscale=.9]
			\node (a) [small vert,fill=orange!50] at (0,1) {$\scriptstyle a$};
			\node (d) [small vert,fill=cyan!40] at (1,0) {$\scriptstyle d$};
			\node (b) [small vert,fill=orange!50] at (1,1) {$\scriptstyle b$};
			\node (c) [small vert,fill=orange!50] at (0,0) {$\scriptstyle c$};
			
			\path [d] (a) to (b);
			\path [d] (a) to (d);
			\path [d] (b) to (d);
			\path [d] (b) to (c);
			\path [d] (c) to (b);
			\path [d] (c) to (d);
			% \path [d] (a) to (c);
			
			\node  at (.5,-.5) {$\mathsf{C}$};   
		\end{tikzpicture}
		\hfil
		%%%%%%%%%%%%%%%%%%%%%%%%%%%%%%%%%%%%%%%%%%%%%%%%%%%
		\begin{tikzpicture}[yscale=.9,xscale=.9]
			\node (x) [small vert] at (1,1) {$\scriptstyle x$};
			\node (y) [small vert,fill=orange!50] at (1,0) {$\scriptstyle y$};
			\node (z) [small vert,fill=cyan!40] at (2,0) {$\scriptstyle z$};
			\node (w) [small vert] at (2,1) {$\scriptstyle w$};
			
			\path [d] (x) to (y);
			\path [d] (y) to (z);
			\path [d] (w) to (y);
			\path [d] (w) to (z);
			\path [d,loop left,distance=7mm] (y) to (y);
			
			\node  at (1.5,-.5) {$\mathsf{D}$};   
		\end{tikzpicture}
		\hfil
		%%%%%%%%%%%%%%%%%%%%%%%%%%%%%%%%%%%%%%%%%%%%%%%%%%%
		\begin{tikzpicture}[xscale=1.5,yscale=1]
			\node (y) [rect vert,,fill=orange!50] at (1,0) {$\scriptstyle\{a,b,c\}$};
			\node (z) [rect vert,fill=cyan!40] at (2,0) {$\scriptstyle \{d\}$};
			
			\path [d] (y) to (z);
			\path [d,loop left,distance=3.8mm] (y) to (y);
			
			\node  at (1.5,-.8) {$\mathsf{C}/\ker(f)$};   
		\end{tikzpicture}
		\caption{\label{fig:1stIso}
			First Isomorphism Theorem example.
		}
	\end{figure}
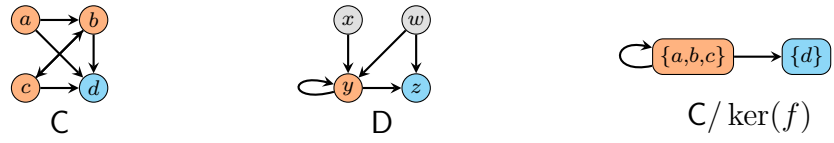
	
	\section{Second Isomorphism Theorem}
	\label{sec:second iso thm}
	
	\begin{lemma}
		\label{lem:cong on h}
		If $\mathsf{C}\leq\mathsf{D}$, then the restriction $\theta_{|\CC}$ of a congruence relation $\theta$ on $\mathsf{D}$ to $\mathsf{C}$ is a congruence relation on $\mathsf{C}$.
	\end{lemma}
	
	\begin{proof}
		First, recall that an equivalence relation restricted to a subset is an equivalence relation on that subset. So $\theta_{|\CC}$ is an equivalence relation on $\mathsf{C}$.  Let $[p]_\DD$ and $[p]_\CC$ denote the equivalence classes of $p$ with respect to $\theta$ and $\theta_{|\CC}$, respectively. If $(p,q)\in\theta_{|\CC}$, then $(p,q)\in\theta$ and so  $[\Opt(p)]_\DD=[\Opt(q)]_\DD$. Hence
		\begin{align*}
			[\Opt(p)]_\CC
			&=\{[r]_\CC \mid r\in \Opt(p)\} =\{[r]_\DD\cap \CC \mid r\in \Opt(p)\} \\
			&=\{[r]_\DD\cap \CC \mid r\in \Opt(q)\} =\{[r]_\CC \mid r\in \Opt(q)\} =[\opt q]_\CC.\qedhere
		\end{align*}
	\end{proof}
	
	\begin{theorem}[Second Isomorphism]\label{thm:second iso}\rm 
		If $\mathsf{C}\leq\mathsf{D}$ and $\theta\in\Con(\DD)$, then $\tilde{\mathsf{C}}=\{[p]\in \DD/\theta \mid [p]\cap \CC\neq \emptyset\}$ is a suboptiongraph of $\mathsf{D}/{\theta}$ isomorphic to $\mathsf{C}/{\theta_{|\CC}}$.  
	\end{theorem}
	
	\begin{proof}
		The restriction $f:\mathsf{C}\to\mathsf{D}/{\theta}$ of the option-preserving quotient map $p\mapsto[p]:\mathsf{D}\to\mathsf{D}/{\theta}$ is also option preserving. Hence the image  $f(\mathsf{C})=\tilde{\mathsf{C}}$ is a suboptiongraph of $\mathsf{D}/{\theta}$. It is clear that $\ker(f)=\theta_{|\CC}$. So the result follows by the First Isomorphism Theorem. 
	\end{proof}
	
	\begin{example}
		\rm 
		Figure~\ref{fig:2ndIso} demonstrates the Second Isomorphism Theorem for an optiongraph $\DD$ with $\CC=\{d,e,g,h,i\}$, $\theta=eg|ihf$, and $\theta_{|\CC}=eg|ih$. 
	\end{example}

	\begin{figure}[h]
		\hfil
		\begin{tikzpicture}[yscale=.7,xscale=1.1]
			\node (a) [small vert] at (0,1.8) {$\scriptstyle a$};
			\node (b) [small vert] at (-1,1) {$\scriptstyle  b$};
			\node (c) [small vert] at (1,1) {$\scriptstyle  c$};
			\node (e) [small vert, text=black] at (0,0) {$\scriptstyle  e$};
			\node (f) [small vert, text=black] at (1,0) {$\scriptstyle f$};
			\node (d) [small vert] at (-1,0) {$\scriptstyle  d$};
			\node (g) [small vert, text=black] at (-1,-1) {$\scriptstyle  g$};
			\node (h) [small vert, text=black] at (1,-1) {$\scriptstyle  h$};
			\node (i) [small vert,text=black] at (0,-1.8) {$\scriptstyle  i$};
			\path [d] (a) to (b);
			\path [d] (a) to (c);
			\path [d] (a) to (e);
			\path [d] (b) to (d);
			\path [d] (b) to (e);
			\path [d] (c) to (f);
			\path [d] (c) to (e);
			\path [d] (d) to (g);
			\path [d] (g) to (i);
			\path [d] (e) to (h);
			\path [d] (e) to (i);
			\path [d,loop right,distance=4mm] (d) to (d);
			
			\begin{pgfonlayer}{background}
				\fill[cyan,opacity=0.3] \convexpath{d,e,h,i,g}{11pt};
			\end{pgfonlayer}
			\node[] at (-1.6,0) {$\mathsf{C}$};
			\node  at (0,-2.8) {$\mathsf{D}$};
		\end{tikzpicture}
		\hfil %%%%%%%%%%%%%%%%%%%%%%%%%%%%%%%%%%%%%%%%%%%%%%%%%%%%%%%%
		\begin{tikzpicture}[yscale=.8,xscale=1.1]
			\node (a) [rect vert] at (0,1.8) {$\scriptstyle \{a\}$};
			\node (b) [rect vert] at (-1,1) {$\scriptstyle  \set{b}$};
			\node (c) [rect vert] at (1,1) {$\scriptstyle  \set{c}$};
			\node (e) [rect vert,fill=orange!50, text=black,minimum size=0.5cm] at (0,-.75) {$\scriptstyle  \set{e,g}$};
			%\node (f) [small vert,fill=pink, text=black,minimum size=0.5cm] at (1,0) {$\scriptstyle  f$};
			\node (d) [rect vert,fill=orange!50] at (-1,0) {$\scriptstyle  \set{d}$};
			%\node (g) [small vert,fill=green, text=black,minimum size=0.5cm] at (-1,-1) {$\scriptstyle  g$};
			\node (h) [rect vert,fill=orange!50, text=black,minimum size=0.5cm] at (1,-1.5) {$\scriptstyle  \set{i,h,f}$};
			%\node (i) [small vert,fill=pink, text=black,minimum size=0.5cm] at (0,-2) {$\scriptstyle  i$};
			\path [d] (a) to (b);
			\path [d] (a) to (c);
			\path [d] (a) to (e);
			\path [d] (b) to (d);
			\path [d] (b) to (e);
			\path [d] (c) to (h);
			\path [d] (c) to (e);
			\path [d] (d) to (e);
			%\path [d] (g) to (i);
			\path [d] (e) to (h);
			%\path [d] (e) to (i);
			\path [d,loop below,distance=6mm] (d) to (d);
			
			% \draw [dashed ,rotate=150,fill=cyan, opacity = 0.2]  (1.35,0.8) arc[start angle=360,end angle=0,x radius=1.8,y radius=1.1];
			\begin{pgfonlayer}{background}
				
			\end{pgfonlayer}
			\node[] at (-1.7,0) {$\tilde{\mathsf{C}}$};
			\node  at (0,-2.2) {$\mathsf{D}/{\theta}$};   
		\end{tikzpicture}
		\hfil %%%%%%%%%%%%%%%%%%%%%%%%%%%%%%%%%%%%%%
		\begin{tikzpicture}[yscale=.8,xscale=1.1]
			% \node (a) [small vert,minimum size=0.5cm] at (0,2) {$\scriptstyle a$};
			%\node (b) [rect vert,minimum size=0.5cm] at (-1,1) {$\scriptstyle  \set{b}$};
			%\node (c) [small vert,minimum size=0.5cm] at (1,1) {$\scriptstyle  c$};
			\node (e) [rect vert,fill=orange!50, text=black] at (0,-.75) {$\scriptstyle  \set{e,g}$};
			%\node (f) [small vert,fill=pink, text=black,minimum size=0.5cm] at (1,0) {$\scriptstyle  f$};
			\node (d) [rect vert,fill=orange!50,minimum size=0.5cm] at (-1,0) {$\scriptstyle  \set{d}$};
			%\node (g) [small vert,fill=green, text=black,minimum size=0.5cm] at (-1,-1) {$\scriptstyle  g$};
			\node (h) [rect vert,fill=orange!50, text=black,minimum size=0.5cm] at (1,-1.5) {$\scriptstyle  \set{i,h}$};
			%\node (i) [small vert,fill=pink, text=black,minimum size=0.5cm] at (0,-2) {$\scriptstyle  i$};
			%\path [d] (a) to (b);
			%\path [d] (a) to (c);
			%\path [d] (a) to (e);
			%\path [d] (b) to (d);
			%\path [d] (b) to (e);
			%\path [d] (c) to (h);
			%\path [d] (c) to (e);
			\path [d] (d) to (e);
			%\path [d] (g) to (i);
			\path [d] (e) to (h);
			\path [d,loop below,distance=6mm] (d) to (d);
			%\path [d] (e) to (i);
			\node  at (0,-2.2) {$\mathsf{C}/{\theta_{|\CC}}$};   
		\end{tikzpicture}
		\caption{\label{fig:2ndIso}
			Second Isomorphism Theorem example.} 
	\end{figure}
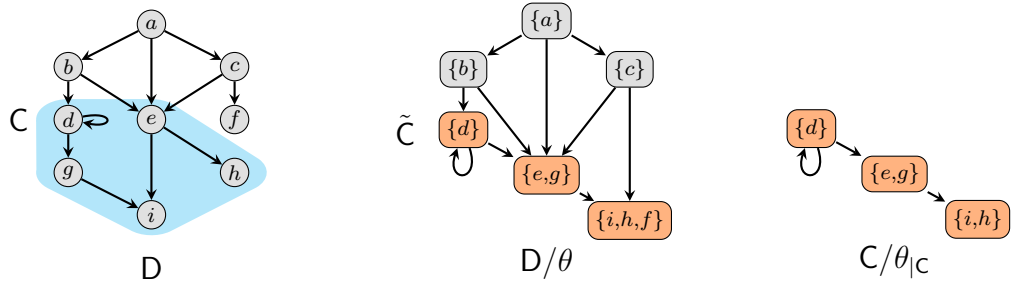
	
	\begin{lemma}
		\label{lem:extension}
		If $\mathsf{C}\leq\mathsf{D}$ and $\theta\in\Con(\CC)$, then the extension $\hat\theta:=\theta\cup\{(p,p)\mid p\in \mathsf{D}\setminus \CC\}\in\Con(\DD)$.
	\end{lemma}
	
	\begin{proof}
		Suppose $p\mathrel{\theta}q$ and $p\ne q$, so that $p,q\in \mathsf{C}$. 
		Then $[\Opt(p)]_{\hat\theta}=[\Opt(p)]_\theta=[\Opt(q)]_\theta=[\Opt(q)]_{\hat\theta}$ since $\CC\le\DD$.
	\end{proof}
	
	\begin{proposition}
		If $\mathsf{C}\leq \mathsf{D}$, then $\CC/{\bowtie_\CC}$ is isomorphic to a suboptiongraph of $\DD/{\bowtie_\DD}$.
	\end{proposition}

    \begin{proof}
	By Lemma~\ref{lem:cong on h}, $({\bowtie}_\DD)_{|\CC}\subseteq {\bowtie}_{\mathsf{C}}$.	By Lemma~\ref{lem:extension}, ${\bowtie}_{\mathsf{C}}=(\widehat\bowtie_{\CC})_{|\CC}\subseteq{(\bowtie}_\DD)_{|\CC}$. Hence ${\bowtie}_{\mathsf{C}}=({\bowtie}_\DD)_{|\CC}$, and so the Second Isomorphism Theorem applied to ${\bowtie}_\DD$ shows that $\CC/{\bowtie_\CC}=\CC/{(\bowtie_\DD)_{|\CC}}$ is isomorphic to the suboptiongraph $\tilde \CC$ of $\DD/{\bowtie_\DD}$.
	\end{proof}
	
	\section{Third Isomorphism Theorem}\label{sec:third iso thm}
	
Suppose  $\eta$ and $\theta$  are  equivalence relations on a set $A$ such that ${\eta} \subseteq {\theta}$. Recall that ${\theta} / {\eta}:= \{(\v{a}_\eta,\v{b}_\eta)\mid (a,b)\in\theta\}$ defines an equivalence relation on $A/{\eta}$. 
	
	\begin{theorem}[Third Isomorphism]\label{thm:third iso}\rm 
		If $\eta,\theta\in\Con(\DD)$ and ${\eta}\subseteq{\theta}$, then ${\theta}/{\eta}$ is a congruence relation on $\mathsf{D}/{\eta}$ and $\l{\mathsf{D}/{\eta}}/\l{{\theta}/{\eta}}$
		is isomorphic to $\mathsf{D}/{\theta}$.
	\end{theorem}
	
	\begin{proof}
		Let $f:\mathsf{D}/{\eta}\to\mathsf{D}/{\theta}$ be defined via $f([p]_\eta):=[p]_\theta$.
		Observe that $f$ is well defined since $(p,q)\in\eta$ implies $(p,q)\in\theta$.
		If $p\in\DD$, then 
\begin{align*}
			\Opt_\theta(f([p]_\eta)) 
			&=\Opt_\theta([p]_\theta) 
			=\{[r]_\theta \mid r\in \Opt(p)\}\\ 
			&=\{f([r]_\eta) \mid r\in \Opt(p)\} 
			=f(\{[r]_\eta \mid r\in \Opt(p)\})\\ 
			&=f(\Opt_\eta([p]_\eta)).
		\end{align*}
		Hence $f$ is option preserving. It is clear that $f$ is surjective and $\ker(f)={\theta}/{\eta}$. The result follows by the First Isomorphism Theorem.
	\end{proof}

	\begin{example} \rm
		Figure~\ref{fig:3rdRule} demonstrates the Third Isomorphism Theorem on a rulegraph $\DD$ with $\eta=ef\subseteq cd|efg=\theta$. Note that ${\bowtie}=ab|cd|efg$.
\end{example}

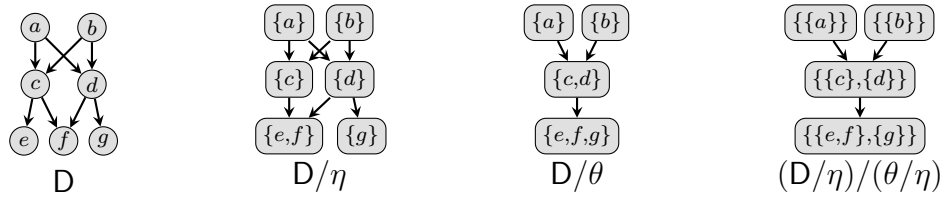
\begin{figure}[h]
	\hfil
	\begin{tikzpicture}[scale=.75]
		\node (02) [small vert] at (0,2) {$\scriptstyle a$};
		\node (12) [small vert] at (1,2) {$\scriptstyle  b$};
		\node (01) [small vert] at (0,1) {$\scriptstyle  c$};
		\node (11) [small vert] at (1,1) {$\scriptstyle  d$};
		\node (00) [small vert] at (-.2,0) {$\scriptstyle  e$};
		\node (10) [small vert] at (.5,0) {$\scriptstyle  f$};
		\node (20) [small vert] at (1.2,0) {$\scriptstyle  g$};
		\path [d] (02) to (01);
		\path [d] (02) to (11);
		\path [d] (12) to (01);
		\path [d] (12) to (11);
		\path [d] (01) to (00);
		\path [d] (01) to (10);
		\path [d] (11) to (10);
		\path [d] (11) to (20);
		\node  at (.5,-.7) {$\DD$};
	\end{tikzpicture}
	\hfil %%%%%%%%%%%%%%%%%%%%%%%%%%%%%%%%
	\begin{tikzpicture}[xscale=.8,yscale=.75]
		\node (02) [rect vert] at (0,2) {$\scriptstyle \{a\}$};
		\node (12) [rect vert] at (1,2) {$\scriptstyle  \{b\}$};
		\node (01) [rect vert] at (0,1) {$\scriptstyle  \{c\}$};
		\node (11) [rect vert] at (1,1) {$\scriptstyle  \{d\}$};
		\node (10) [rect vert] at (0,0) {$\scriptstyle  \{e,f\}$};
		\node (20) [rect vert] at (1.2,0) {$\scriptstyle  \{g\}$};
		\path [d] (02) to (01);
		\path [d] (02) to (11);
		\path [d] (12) to (01);
		\path [d] (12) to (11);
		%\path [d] (01) to (00);
		\path [d] (01) to (10);
		\path [d] (11) to (10);
		\path [d] (11) to (20);
		\node  at (.5,-.7) {$\DD/{\eta}$};
	\end{tikzpicture}
	\hfil %%%%%%%%%%%%%%%%%%%%%%%%%%%%%%%%%%%
	\begin{tikzpicture}[scale=.75]
		\node (02) [rect vert] at (0,2) {$\scriptstyle \{a\}$};
		\node (12) [rect vert] at (1,2) {$\scriptstyle  \{b\}$};
		\node (01) [rect vert] at (0.5,1) {$\scriptstyle  \{c,d\}$};
		\node (10) [rect vert] at (.5,0) {$\scriptstyle  \{e,f,g\}$};
		\path [d] (02) to (01);
		\path [d] (12) to (01);
		\path [d] (01) to (10);
		\node  at (.5,-.7) {$\mathsf{D}/{\theta}$};
	\end{tikzpicture}
	\hfil %%%%%%%%%%%%%%%%%%%%%%%%%%%%%%%%%%
	\begin{tikzpicture}[xscale=.8,yscale=.75]
		\node (02) [rect vert] at (-.15,2) {$\scriptstyle \{\{a\}\}$};
		\node (12) [rect vert] at (1.15,2) {$\scriptstyle  \{\{b\}\}$};
		\node (01) [rect vert] at (0.5,1) {$\scriptstyle  \{\{c\},\{d\}\}$};
		\node (10) [rect vert] at (.5,0) {$\scriptstyle  \{\{e,f\},\{g\}\}$};
		\path [d] (02) to (01);
		\path [d] (12) to (01);
		\path [d] (01) to (10);
		\node  at (.5,-.7) {$(\mathsf{D}/{\eta})/({\theta}/{\eta})$};
	\end{tikzpicture}
	\caption{
		\label{fig:3rdRule}
		Third Isomorphism Theorem example for a rulegraph.}
\end{figure}

\begin{example} \rm
	Figure~\ref{fig:3rdLoopy} demonstrates the Third Isomorphism Theorem on a cyclic optiongraph $\DD$ with $\eta=xy\subseteq xyz|ab=\theta$. Note that $F=\{x,y,z\}$, $M=\{a,b,c\}$, $I=\emptyset$, and ${\bowtie}=abc|xyz$.
\end{example}

\begin{figure}[h]
	\hfil
	\begin{tikzpicture}[xscale=.9,yscale=.9]
		\node (x) [small vert,fill=green!40] at (0,1) {$\scriptstyle x$};
		\node (y) [small vert,fill=green!40] at (0,0) {$\scriptstyle y$};
		\node (a) [small vert] at (1,1) {$\scriptstyle a$};
		\node (b) [small vert] at (1,0) {$\scriptstyle b$};
		\node (c) [small vert] at (2,0.5) {$\scriptstyle  c$};
		\node (z) [small vert,fill=green!40] at (3,0.5) {$\scriptstyle z$};
		\path [d] (a) to (x);
		\path [d] (b) to (y);
		\path [d] (c) to (z);
		\path [d] (a) to (b);
		\path [d] (b) to (a);
		\path [d] (a) to (c);
		\path [d] (b) to (c);
		\path [d,loop above,distance=7mm, looseness=2] ([xshift=-1ex]c.north) to ([xshift=1ex]c.north);
		% \path [d,loop above] (c) to (c);
		% \path [d] (c) to (1.8,1.2) to (2.2,1.2) to (c);
		\node  at (1,-.7) {$\mathsf{D}$};
	\end{tikzpicture}
	\hfil %%%%%%%%%%%%%%%%%%%%%%%%%%%%%%%%%%%%%%%%%%%%%%
	\begin{tikzpicture}[xscale=1.1,yscale=.9]
		\node (xy) [rect vert,fill=green!40] at (0,.5) {$\scriptstyle \{x,y\}$};
		\node (a) [rect vert] at (1,1) {$\scriptstyle \{a\}$};
		\node (b) [rect vert] at (1,0) {$\scriptstyle \{b\}$};
		\node (c) [rect vert] at (2,0.5) {$\scriptstyle  \{c\}$};
		\node (z) [rect vert,fill=green!40] at (3,0.5) {$\scriptstyle \{z\}$};
		\path [d] (a) to (xy);
		\path [d] (b) to (xy);
		\path [d] (c) to (z);
		\path [d] (a) to (b);
		\path [d] (b) to (a);
		\path [d] (a) to (c);
		\path [d] (b) to (c);
		\path [d,loop above,distance=7mm] ([xshift=-1ex]c.north) to ([xshift=1ex]c.north);
		% \path [d] (c) to (1.8,1.2) to (2.2,1.2) to (c);
		\node  at (1.5,-.7) {$\mathsf{D}/{\eta}$};
	\end{tikzpicture}
	\hfil %%%%%%%%%%%%%%%%%%%%%%%%%%%%%%%%%%%%%%%%%%%%%%
	\begin{tikzpicture}[xscale=1.25,yscale=.9]
		\node (xyz) [rect vert,fill=green!40] at (0,.5) {$\scriptstyle \{x,y,z\}$};
		\node (ab) [rect vert] at (1,1) {$\scriptstyle \{a,b\}$};
		\node (c) [rect vert] at (1,0) {$\scriptstyle  \{c\}$};
		\path [d] (ab) to (xyz);
		\path [d] (c) to (xyz);
		\path [d] (ab) to (c);
		\path [d] (ab) to (c);
		\path [d,loop right,distance=7mm] (ab) to (ab);
		% \path [d] (ab) to (1.8,1.2) to (1.8,.8) to (ab);
		% \path [d,out=25,in=-25,looseness=6] ([yshift=-1ex]c.east) to ([yshift=1ex]c.east);
		\path [d,loop right,distance=7mm] (c) to (c);
		% \path [d] (c) to (1.8,0.2) to (1.8,-.2) to (c);
		\node  at (1,-.7) {$\mathsf{D}/{\theta}$};
	\end{tikzpicture}
	\hfil %%%%%%%%%%%%%%%%%%%%%%%%%%%%%%%%%%%%%%%%%%%%%%
	\begin{tikzpicture}[xscale=1.6,yscale=.9]
		\node (xyz) [rect vert,fill=green!40] at (-.15,.5) {$\scriptstyle \{\{x,y\},\{z\}\}$};
		\node (ab) [rect vert] at (1,1) {$\scriptstyle \{\{a\},\{b\}\}$};
		\node (c) [rect vert] at (1,0) {$\scriptstyle  \{\{c\}\}$};
		\path [d] (ab) to (xyz);
		\path [d] (c) to (xyz);
		\path [d] (ab) to (c);
		\path [d] (ab) to (c);
		\path [d,loop right,distance=7mm] (ab) to (ab);
		% \path [d] (ab) to (1.8,1.2) to (1.8,.8) to (ab);
		\path [d,loop right,distance=7mm] (c) to (c);
		% \path [d] (c) to (1.8,0.2) to (1.8,-.2) to (c);
		\node  at (.5,-.7) {$(\mathsf{D}/{\eta})/({\theta}/{\eta})$};
	\end{tikzpicture}
	\caption{\label{fig:3rdLoopy}
		Third Isomorphism Theorem example for a cyclic optiongraph.}
\end{figure}
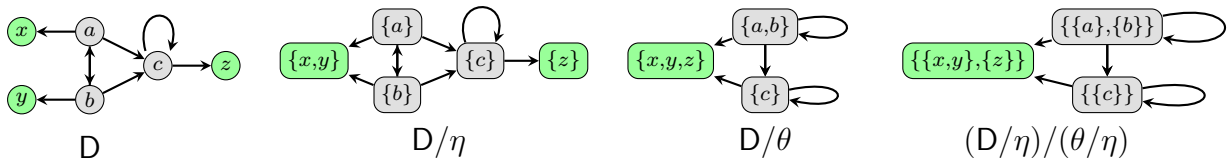

\section{Fourth Isomorphism Theorem}\label{sec:fourth iso thm}

The following result for rulegraphs originally appeared in~\cite{Basic2024}, but the proof was left to the reader. We state this result for optiongraphs and include the proof.
\begin{theorem}[Fourth Isomorphism]\label{thm:fourth iso}\rm 
	If $\mathsf{D}$ is an optiongraph and $\theta\in\Con(\mathsf{D})$, then the interval $[\theta,\bowtie]=\{{\phi}\in\Con(\mathsf{D})\mid \theta\subseteq \phi \}$ is a sublattice of $\Con(\mathsf{D})$ and
	\[
	\alpha:[\theta,\bowtie]\to\Con(\mathsf{D}/\theta) 
	\]
	defined by $\alpha(\phi):=\phi/\theta$ is a lattice isomorphism.
\end{theorem}

\begin{proof}
	Every interval of a lattice is a sublattice, so $[\theta,\bowtie]$ is a sublattice. To show that $\alpha$ is injective, assume that $\alpha(\phi)=\alpha(\psi)$, so that  $\phi/\theta=\psi/\theta$. Then $\phi=\psi$ since
	\[
	(p,q)\in\phi \Leftrightarrow ([p]_\theta,[q]_\theta)\in\phi/\theta \Leftrightarrow ([p]_\theta,[q]_\theta)\in\psi/\theta \Leftrightarrow (p,q)\in\psi.
	\]
	
	To show that $\alpha$ is surjective, suppose $\psi\in \Con(\mathsf{D}/\theta)$. 
	Define the canonical quotient maps $f_\theta:\mathsf{D}\to \mathsf{D}/\theta$ and $f_\psi: \mathsf{D}/\theta\to  (\mathsf{D}/\theta)/\psi$, and let $\phi:=\ker(f_\psi\circ f_\theta)$. So $\phi$ is a congruence relation on $\mathsf{D}$ by Theorem~\ref{thm:first iso}. It is easy to see that $\phi=\{(p,q)\mid ([p]_\theta,[q]_\theta)\in \psi\}$ and $\theta\subseteq\phi$. Then for $p,q\in \DD$,
	\[
	([p]_\theta,[q]_\theta)\in\phi/\theta \Leftrightarrow (p,q)\in \phi \Leftrightarrow ([p]_\theta,[q]_\theta)\in\psi.   
	\]
	So $\psi=\phi/\theta=\alpha(\phi)$.
	
	Finally, $\alpha$ is an isomorphism since
	\[
	\phi\subseteq\psi \Leftrightarrow \phi/\theta\subseteq\psi/\theta\Leftrightarrow\alpha(\phi)\subseteq \alpha(\psi)
	\]
	for all $\phi,\psi\in[\theta,\bowtie]$.
\end{proof}

\begin{example} \rm
	Figure~\ref{fig:4thIsoLoopy} demonstrates the Fourth Isomorphism Theorem on a cyclic optiongraph $\mathsf{D}$ with $\theta=abc$. Note that $\mathsf{I}_\DD=\DD$ and ${\bowtie}=abcxy$ identifies all positions.
\end{example}

\begin{figure}[h]
	\hfil
	\begin{tikzpicture}[xscale=1.1]
		\node (x) [small vert] at (1,1) {$\scriptstyle x$};
		\node (a) [small vert] at (1,0) {$\scriptstyle a$};
		\node (b) [small vert] at (2,0) {$\scriptstyle b$};
		\node (y) [small vert] at (2,1) {$\scriptstyle  y$};
		\node (c) [small vert] at (1.5,1) {$\scriptstyle  c$};
		\path [d] (x) to (a);
		\path [d] (b) to (c);
		\path [d] (c) to (a);
		\path [d] (a) to (b);
		\path [d] (y) to (b);
		\node  at (1.5,-.7) {$\mathsf{D}$};
	\end{tikzpicture}
	\hfil %%%%%%%%%%%%%%%%%%%%%%%%%%%%%%%%%%%%%%%%%%%%%%%%%%%%%%%%%%%%%%
	\begin{tikzpicture}[xscale=1,yscale=.65]
		\node (abcxy) [rect vert, fill=orange!50] at (-.5,3) {$\scriptstyle abcxy$};
		\node (ay cx) [rect vert] at (-2,2) {$\scriptstyle ay|cx$};
		\node (abcx) [rect vert, fill=orange!50] at (-1,2) {$\scriptstyle abcx$};
		\node (abcy) [rect vert, fill=orange!50] at (0,2) {$\scriptstyle abcy$};
		\node (abc xy) [rect vert,fill=orange!50] at (1,2) {$\scriptstyle  abc|xy$};
		\node (abc) [rect vert,fill=orange!50] at (0,1) {$\scriptstyle abc $};
		\node (ay) [rect vert] at (-1,1) {$\scriptstyle ay $};
		\node (cx) [rect vert] at (-2,1) {$\scriptstyle cx $};
		\node (none) [rect vert] at (-1,0) {$\scriptstyle  $};
		\path [d,-] (abcxy) to (ay cx);
		\path [d,-] (abcxy) to (abcx);
		\path [d,-] (abcxy) to (abcy);
		\path [d,-] (abcxy) to (abc xy);
		\path [d,-] (abcx) to (abc);
		\path [d,-] (abcy) to (abc);
		\path [d,-] (abc xy) to (abc);
		\path [d,-] (ay cx) to (ay);
		\path [d,-] (abcx) to (cx);
		\path [d,-] (ay cx) to (cx);
		\path [d,-] (abcy) to (ay);
		\path [d,-] (cx) to (none);
		\path [d,-] (ay) to (none);
		\path [d,-] (abc) to (none);
		\node  at (-0.5,-.8) {$\Con(\mathsf{D})$};
	\end{tikzpicture}
	\hfil %%%%%%%%%%%%%%%%%%%%%%%%%%%%%%%%%%%%%%%%%%%%%%%%%%%%%%%%%%%%%%
	\begin{tikzpicture}[xscale=.9]
		\node (x) [rect vert] at (0,1) {$\scriptstyle \{x\}$};
		\node (y) [rect vert] at (2,1) {$\scriptstyle  \{y\}$};
		\node (abc) [rect vert] at (1,0) {$\scriptstyle  \{a,b,c\}$};
		\path [d] (x) to (abc);
		\path [d] (y) to (abc);
		% \path [d] (abc) to (.8,1) to (1.2,1) to (abc);
		\path [d,in=110,out=70,loop,looseness=12] (abc) to (abc);
		\node  at (1,-.7) {$\mathsf{D}/\theta$};
	\end{tikzpicture}
	\hfil %%%%%%%%%%%%%%%%%%%%%%%%%%%%%%%%%%%%%%%%%%%%%%%%%%%%%%%%%%%%%%
	\begin{tikzpicture}[xscale=1.5,yscale=.65]
		\node (abc x y) [rect vert, fill=orange!50] at (0,3) {$\scriptstyle \{a,b,c\}\{x\}\{y\}$};
		\node (abc x) [rect vert, fill=orange!50] at (-1,2) {$\scriptstyle \{a,b,c\}\{x\}$};
		\node (abc y) [rect vert, fill=orange!50] at (1,2) {$\scriptstyle \{a,b,c\}\{y\}$};
		\node (x y) [rect vert,fill=orange!50] at (0,2) {$\scriptstyle  \{x\}\{y\}$};
		\node (id) [rect vert,fill=orange!50] at (0,1) {$\scriptstyle  $};
		\path [d,-] (abc x y) to (abc x);
		\path [d,-] (abc x y) to (abc y);
		\path [d,-] (abc x y) to (x y);
		\path [d,-] (abc x) to (id);
		\path [d,-] (abc y) to (id);
		\path [d,-] (x y) to (id);
		\node  at (0,-.8) {$\Con(\mathsf{D}/\theta)$};
	\end{tikzpicture}
	
	\caption{
		\label{fig:4thIsoLoopy}
		Fourth Isomorphism Theorem example for a cyclic optiongraph. The highlighted portion of the second diagram is the interval $[\theta,\bowtie]$.}
\end{figure}
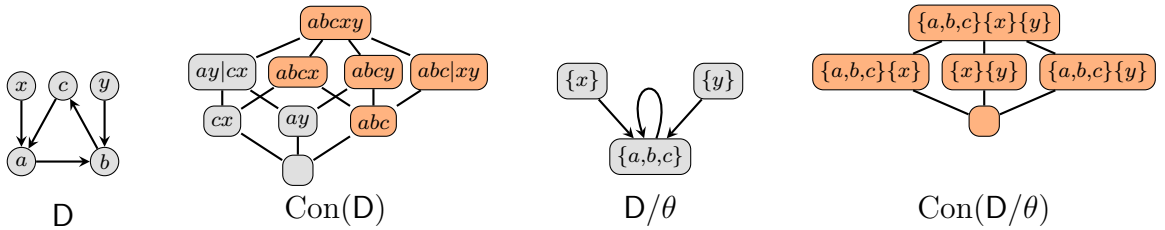

\begin{example} \rm
	Figure~\ref{fig:4thIsoInfinite} demonstrates the Fourth Isomorphism Theorem on an infinite rulegraph $\DD$ with $\theta=01|23$. Note that  ${\bowtie}=01|23|45|\ldots$ and $\mathsf{F}_\DD=\DD$. 
\end{example}

\begin{figure}[h]
	\hfil
	\begin{tikzpicture}[scale=1]
		\node (03) at (0,3) {$\scriptstyle \vdots$};
		\node (13) at (1,3) {$\vdots$};
		\node (02) [small vert] at (0,2) {$\scriptstyle 4$};
		\node (12) [small vert] at (1,2) {$\scriptstyle  5$};
		\node (01) [small vert] at (0,1) {$\scriptstyle  2$};
		\node (11) [small vert] at (1,1) {$\scriptstyle  3$};
		\node (00) [small vert] at (0,0) {$\scriptstyle  0$};
		\node (10) [small vert] at (1,0) {$\scriptstyle  1$};
		
		\path [d] (02) to (01);
		\path [d] (02) to (11);
		\path [d] (12) to (11);
		\path [d] (01) to (00);
		\path [d] (11) to (10);
		\path [d] (11) to (00);
		\path [d] (03) to (02);
		\path [d] (13) to (12);
		\path [d] (13) to (02);
		\node  at (0.5,-.6) {$\DD$};
	\end{tikzpicture}
	%%%%%%%%%%%%%%%%%%%%%%%%%%%
	\hfil
	\begin{tikzpicture}[scale=.6]
		\node (all) [rect vert, fill=orange!50] at (0,4) {$\scriptstyle 01|23|45|\ldots$};
		\node (vdots) at (0,3) {$\phantom{f} $};
		\node at (0,3.15) {$\scriptstyle \vdots$};
		\node (c01-23-45) [rect vert, fill=orange!50] at (0,2) {$\scriptstyle  01|23|45$};
		\node (c01-23) [rect vert,fill=orange!50] at (0,1) {$\scriptstyle  01|23$};
		\node (c01) [rect vert] at (0,0) {$\scriptstyle  01$};
		\node (id0) [rect vert] at (0,-1) {$\scriptstyle  $};
		
		% \path [d,-] (all) to (vdots);
		\path [d,-] (vdots) to (c01-23-45);
		\path [d,-] (c01-23-45) to (c01-23);
		\path [d,-] (c01-23) to (c01);
		\path [d,-] (c01) to (id0);
		\node  at (0,-1.8) {$\Con(\DD)$};
	\end{tikzpicture}
	%%%%%%%%%%%%%%%%%%%%%%%%%%%%
	\hfil
	\begin{tikzpicture}[scale=1]
		\node (03)  at (0,3) {$\scriptstyle \vdots$};
		\node (13)  at (1,3) {$\vdots$};
		\node (02) [rect vert] at (0,2) {$\scriptstyle \{4\}$};
		\node (12) [rect vert] at (1,2) {$\scriptstyle  \{5\}$};
		\node (05 1)[rect vert] at (0.5,1) {$\scriptstyle  \{2,3\}$};
		\node (05 0) [rect vert] at (0.5,0) {$\scriptstyle  \{0,1\}$};
		
		\path [d] (02) to (05 1);
		\path [d] (05 1) to (05 0);
		\path [d] (12) to (05 1);
		\path [d] (03) to (02);
		\path [d] (13) to (12);
		\path [d] (13) to (02);
		
		\node  at (0.5,-.6) {$\DD/\theta$};
	\end{tikzpicture}
	%%%%%%%%%%%%%%%%%%%%%%%%%%%%%%%
	\hfil
	\begin{tikzpicture}[scale=.6]
		\node (all) [rect vert, fill=orange!50] at (0,4) {$\scriptstyle \{4\}\{5\}|\ldots$};
		\node (vdots) at (0,3) {$\phantom{f} $};
		\node at (0,3.15) {$\scriptstyle \vdots$};
		\node (c45) [rect vert, fill=orange!50] at (0,2) {$\scriptstyle  \{4\}\{5\}$};
		\node (id0) [rect vert,fill=orange!50] at (0,1) {$\scriptstyle  $};
		% \path [d,-] (all) to (vdots);
		\path [d,-] (vdots) to (c01-23-45);
		\path [d,-] (c45) to (id0);
		\node  at (0,-1.8) {$\Con(\DD/\theta)$};
	\end{tikzpicture}
	\caption{
		\label{fig:4thIsoInfinite}
		Fourth Isomorphism Theorem example for an infinite rulegraph.
	}
\end{figure}
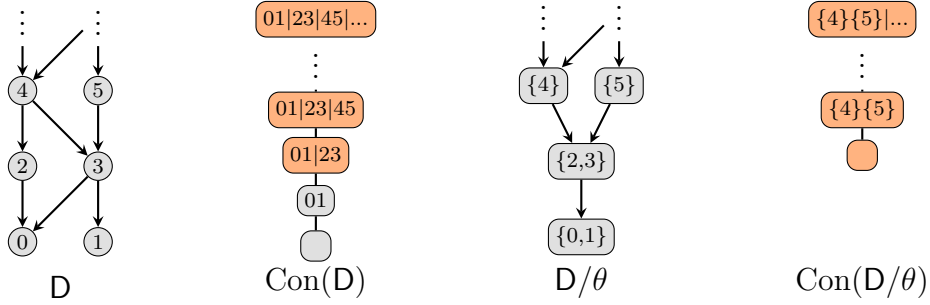

\section{Applications}
\label{sec:applications}

\subsection{Subcategories} 

Rulegraphs form a full subcategory {\bf RGph} of {\bf OGph} by~\cite[Proposition~4.22]{Basic2024}. Optiongraphs with an infinite play also form a full subcategory of {\bf OGph} because the image of such an optiongraph through an option-preserving map has an infinite play, as well. A consequence is that the four isomorphism theorems for {\bf OGph} also hold in both of these subcategories.

\subsection{Simplicity}\label{sec:simplicity}

The maximum element of $\Con(\DD)$ is $\bowtie$ while the minimum is the trivial equivalence relation. An optiongraph is \emph{simple} if $\bowtie$ is trivial, that is, $\DD$ has only the trivial congruence relation. In this case, $\Con(\DD)$ is trivial. The reader should not confuse our notion of simple with the graph-theoretic notion.

\begin{proposition}
The minimum quotient $\mathsf{D}/{\bowtie}$ is the unique simple quotient of the optiongraph $\mathsf{D}$.
\end{proposition}

\begin{proof}
	The quotient $\mathsf{D}/\theta$ is simple if and only if $\Con(\mathsf{D}/\theta)$ is trivial. By the Fourth Isomorphism Theorem, this happens exactly when $\theta={\bowtie}$. 
\end{proof}

\begin{example}
\rm
There are 2 simple optiongraphs with 1 position and 3 with 2 positions. The 15 simple optiongraphs with 3 positions are shown in  Figure~\ref{fig:simpleSize3}. Observe that it is possible for a simple cyclic optiongraph to be disconnected. This does not happen in the case of rulegraphs since all terminal positions are identified in the minimum quotient. Computer calculations using \cite{digraphs} show that there are 289, 19787, and 4537065 simple optiongraphs with 4, 5, and 6 positions, respectively. The sequence that counts simple optiongraphs with $n$ positions is new to the OEIS~\cite{oeis}.
\end{example}

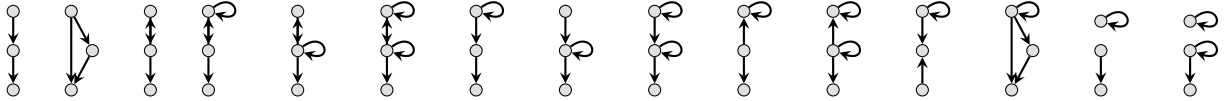
\begin{figure}[h]
	\hfil
	\begin{tikzpicture}[yscale=.52,rotate=-90]
		\node (a) [tiny vert] at (0,0) {};
		\node (b) [tiny vert] at (1,0) {};
		\node (c) [tiny vert] at (2,0) {};
		\path [d] (a) to (b);
		
		\path [d] (b) to (c);
	\end{tikzpicture}
	\hfil
	%%%%%%%%%%%%%%%%%%%%%%l
	\begin{tikzpicture}[yscale=.52,rotate=-90]
		\node (a) [tiny vert] at (0,0) {};
		\node (b) [tiny vert] at (1,0.28) {};
		\node (c) [tiny vert] at (2,0) {};
		\path [d] (a) to (b);
		\path [d] (b) to (c);
		\path [d] (a) to (c);
	\end{tikzpicture}
	\hfil
	%%%%%%%%%%%%%%%%%%%%%%
	\begin{tikzpicture}[yscale=.52,rotate=-90]
		\node (a) [tiny vert] at (0,0) {};
		\node (b) [tiny vert] at (1,0) {};
		\node (c) [tiny vert] at (2,0) {};
		\path [d] (a) to (b);
		\path [d] (b) to (a);
		\path [d] (b) to (c);
	\end{tikzpicture}
	\hfil
	%%%%%%%%%%%%%%%%%%%%%%
	\begin{tikzpicture}[yscale=.52,rotate=-90]
		\node (a) [tiny vert] at (0,0) {};
		\node (b) [tiny vert] at (1,0) {};
		\node (c) [tiny vert] at (2,0) {};
		\path [d] (a) to (b);
		\path [d] (b) to (a);
		\path [d] (b) to (c);
		\path [d,loop above,distance=6mm,in=45,out=135] (a) to (a);
	\end{tikzpicture}
	\hfil
	%%%%%%%%%%%%%%%%%%%%%%
	\begin{tikzpicture}[yscale=.52,rotate=-90]
		\node (a) [tiny vert] at (0,0) {};
		\node (b) [tiny vert] at (1,0) {};
		\node (c) [tiny vert] at (2,0) {};
		\path [d] (a) to (b);
		\path [d] (b) to (a);
		\path [d] (b) to (c);
		\path [d,loop above,distance=6mm,in=45,out=135] (b) to (b);
	\end{tikzpicture}
	\hfil
	%%%%%%%%%%%%%%%%%%%%%%
	\begin{tikzpicture}[yscale=.52,rotate=-90]
		\node (a) [tiny vert] at (0,0) {};
		\node (b) [tiny vert] at (1,0) {};
		\node (c) [tiny vert] at (2,0) {};
		\path [d] (a) to (b);
		\path [d] (b) to (a);
		\path [d] (b) to (c);
		\path [d,loop above,distance=6mm,in=45,out=135] (a) to (a);
		\path [d,loop above,distance=6mm,in=45,out=135] (b) to (b);
	\end{tikzpicture}
	\hfil
	%%%%%%%%%%%%%%%%%%%%%%
	\begin{tikzpicture}[yscale=.52,rotate=-90]
		\node (a) [tiny vert] at (0,0) {};
		\node (b) [tiny vert] at (1,0) {};
		\node (c) [tiny vert] at (2,0) {};
		\path [d] (a) to (b);
		\path [d] (b) to (c);
		\path [d,loop above,distance=6mm,in=45,out=135] (a) to (a);
	\end{tikzpicture}
	\hfil
	%%%%%%%%%%%%%%%%%%%%%%
	\begin{tikzpicture}[yscale=.52,rotate=-90]
		\node (a) [tiny vert] at (0,0) {};
		\node (b) [tiny vert] at (1,0) {};
		\node (c) [tiny vert] at (2,0) {};
		\path [d] (a) to (b);
		\path [d] (b) to (c);
		\path [d,loop above,distance=6mm,in=45,out=135] (b) to (b);
	\end{tikzpicture}
	\hfil
	%%%%%%%%%%%%%%%%%%%%%%
	\begin{tikzpicture}[yscale=.52,rotate=-90]
		\node (a) [tiny vert] at (0,0) {};
		\node (b) [tiny vert] at (1,0) {};
		\node (c) [tiny vert] at (2,0) {};
		\path [d] (a) to (b);
		\path [d] (b) to (c);
		\path [d,loop above,distance=6mm,in=45,out=135] (a) to (a);
		\path [d,loop above,distance=6mm,in=45,out=135] (b) to (b);
	\end{tikzpicture}
	\hfil
	%%%%%%%%%%%%%%%%%%%%%%
	\begin{tikzpicture}[yscale=.52,rotate=-90]
		\node (a) [tiny vert] at (0,0) {};
		\node (b) [tiny vert] at (1,0) {};
		\node (c) [tiny vert] at (2,0) {};
		\path [d] (b) to (a);
		\path [d] (b) to (c);
		\path [d,loop above,distance=6mm,in=45,out=135] (a) to (a);
	\end{tikzpicture}
	\hfil
	%%%%%%%%%%%%%%%%%%%%%%
	\begin{tikzpicture}[yscale=.52,rotate=-90]
		\node (a) [tiny vert] at (0,0) {};
		\node (b) [tiny vert] at (1,0) {};
		\node (c) [tiny vert] at (2,0) {};
		\path [d] (b) to (a);
		\path [d] (b) to (c);
		\path [d,loop above,distance=6mm,in=45,out=135] (a) to (a);
		\path [d,loop above,distance=6mm,in=45,out=135] (b) to (b);
	\end{tikzpicture}
	\hfil
	%%%%%%%%%%%%%%%%%%%%%%
	\begin{tikzpicture}[yscale=.52,rotate=-90]
		\node (a) [tiny vert] at (0,0) {};
		\node (b) [tiny vert] at (1,0) {};
		\node (c) [tiny vert] at (2,0) {};
		\path [d] (a) to (b);
		\path [d] (c) to (b);
		\path [d,loop above,distance=6mm,in=45,out=135] (a) to (a);
	\end{tikzpicture}
	\hfil
	%%%%%%%%%%%%%%%%%%%%%%
	\begin{tikzpicture}[yscale=.52,rotate=-90]
		\node (a) [tiny vert] at (0,0) {};
		\node (b) [tiny vert] at (1,0.28) {};
		\node (c) [tiny vert] at (2,0) {};
		\path [d] (a) to (b);
		\path [d] (b) to (c);
		\path [d] (a) to (c);
		\path [d,loop above,distance=6mm,in=45,out=135] (a) to (a);
	\end{tikzpicture}
	\hfil
	%%%%%%%%%%%%%%%%%%%%%%
	\begin{tikzpicture}[yscale=.52,rotate=-90]
		\node (a) [tiny vert] at (0.25,0) {};
		\node (b) [tiny vert] at (1,0) {};
		\node (c) [tiny vert] at (2,0) {};
		\path [d] (b) to (c);
		\path [d,loop above,distance=6mm,in=45,out=135] (a) to (a);
	\end{tikzpicture}
	\hfil
	%%%%%%%%%%%%%%%%%%%%%%
	\begin{tikzpicture}[yscale=.52,rotate=-90]
		\node (a) [tiny vert] at (0.25,0) {};
		\node (b) [tiny vert] at (1,0) {};
		\node (c) [tiny vert] at (2,0) {};
		\path [d] (b) to (c);
		\path [d,loop above,distance=6mm,in=45,out=135] (a) to (a);
		\path [d,loop above,distance=6mm,in=45,out=135] (b) to (b);
	\end{tikzpicture}
	
	\caption{
		\label{fig:simpleSize3}
		All  simple optiongraphs with 3 positions.
	}
\end{figure}

\begin{question}\rm
	Is there a reasonable enumeration of simple optiongraphs with $n$ positions?
\end{question}

\begin{proposition} \label{prop:min quot of a quot}
	If $\theta\in\Con(\DD)$, then the minimum quotient of $\mathsf{D}$ is isomorphic to the minimum quotient of $\mathsf{D}/{\theta}$.
\end{proposition}

\begin{proof}
	The Third Isomorphism Theorem implies that
	$\mathsf{D}/{\bowtie}\cong(\mathsf{D}/{\theta})/({\bowtie}/{\theta})$ since ${\theta}\subseteq{\bowtie}$.
\end{proof}

\begin{corollary}\label{cor:iso quotients, iso min quotients}\rm 
	If some quotients of $\mathsf{D}$ and $\mathsf{S}$ are isomorphic, then their minimum quotients $\mathsf{D}/{\bowtie}$ and $\SS/{\bowtie}$ are also isomorphic.
\end{corollary}

A rulegraph is simple if and only if the option map is injective~\cite[Proposition~7.3]{Basic2024}. This is not true for optiongraphs. In fact, in a directed cycle, no two positions have the same option set but the digraph is not simple as the next result shows.

\begin{proposition} \rm
	\label{prop:smash cannot reach terminal}
	If $p,q\in I$, then $p\bowtie q$.
\end{proposition}

\begin{proof}
	Let $\theta$ be the relation that identifies the positions in $I$. If $r\in I$, then $\emptyset\ne\Opt(r)\subseteq I$, and hence $[\Opt(r)]=\{I\}$. Thus $\theta$ is a congruence relation, and so $(p,q)\in\theta\subseteq{\bowtie}$.
\end{proof}

\begin{proposition}
\label{prop:Iclass}
If $p\in I$, then $[p]_{\bowtie}=I$.
\end{proposition}

\begin{proof}
The Second Isomorphism Theorem applied to $\mathsf{I}$ and $\bowtie$ implies that $\tilde{\mathsf{I}}$ is isomorphic to $\mathsf{I}/{\bowtie_{|\mathsf{I}}}$, which is a single position with a loop. So $\Opt_{\bowtie}([p]_{\bowtie})=\{[p]_{\bowtie}\}$.
If $t$ is a terminal position, then $\Opt_{\bowtie}([t]_{\bowtie})=[\Opt(t)]_{\bowtie}=[\emptyset]_{\bowtie}=\emptyset$. Hence $t\notin [p]_{\bowtie}$.

We know that $I\subseteq [p]_{\bowtie}$ by Proposition~\ref{prop:smash cannot reach terminal}. 
Suppose $q\in [p]_{\bowtie}\setminus I$. Then there is a finite walk from $q$ to a terminal position $t$. Along this walk we can find two positions $r\in [q]_{\bowtie}$ and $s\in\Opt(r)\setminus [q]_{\bowtie}$. This gives the contradiction
\[
[p]_{\bowtie}=[q]_{\bowtie}\ne[s]_{\bowtie}\in[\Opt(r)]_{\bowtie}=\Opt_{\bowtie}([r]_{\bowtie})=\Opt_{\bowtie}([p]_{\bowtie})=\{[p]_{\bowtie}\}.
\qedhere
\]
\end{proof}

The following easy consequence is a stronger form of an observation of~\cite[p. 181]{Li}. 

\begin{corollary}\label{cor:black hole}\rm 
	An optiongraph $\mathsf{D}$ has no terminal position if and only if $\mathsf{D}/{\bowtie}$ is a one-vertex loop.
\end{corollary}

It is clear that positions in $I$ have infinite nim-values. The converse is false as seen in Example \ref{ex:extendednims}. 

\begin{question}\rm 
	What families of known games are simple? For example, we conjecture that Fair Shares and Varied Pairs~\cite{ww2} is simple with any number of almonds.
\end{question}

\subsection{Sums}
\label{sec:sums}

The \emph{sum} $\CC+\DD$ of two optiongraphs $\CC$ and $\DD$ is the digraph box product with $\Opt_{\CC+\DD}(p,q)=(\Opt_\CC(p)\times\{q\})\cup(\{p\}\times\Opt_\DD(q))$. This construction matches the classical game sum. 

\begin{proposition} 
	\label{prop:product map}
	If $f:\mathsf{C}\to \tilde{\mathsf{C}}$ and $g:\mathsf{D}\to \tilde{\mathsf{D}}$ are option preserving, then $f\times g:\mathsf{C}+\mathsf{D}\to\tilde{\mathsf{C}}+\tilde{\mathsf{D}}$ defined by $(f\times g)(c,d):=(f(c), g(d))$  is also option preserving.
\end{proposition}

\begin{proof}
	The computation
	\begin{align*}
		(f\times g)(\Opt(r,s))
		&= (f\times g)((\Opt(r)\times\{s\}) \cup 
		(\{r\}\times\Opt(s))) \\
		&= (f(\Opt(r))\times\{ g(s)\}) \cup 
		(\{f(r)\}\times g(\Opt(s))) \\                       
		&= (\Opt(f(r))\times\{ g(s)\}) \cup  
		(\{f(r)\}\times\Opt( g(s)))\\ 
		&= \Opt(f(r), g(s)) 
		= \Opt((f\times g)(r,s))
	\end{align*}
	verifies the result. 
\end{proof}

\goodbreak

The following is an immediate consequence of Proposition~\ref{prop:product map} after applying the First Isomorphism Theorem. 

\begin{corollary}\label{cor: quotient of a sum}\rm 
	If $\phi$ and $\psi$ are congruence relations on the optiongraphs $\mathsf{C}$ and $\mathsf{D}$, respectively, then 
	\[
	\theta:=\{((c,d),(c',d')) \mid c\mathrel{\phi} c' \text{ and } d \mathrel{\psi}d'\}
	\]
	is a congruence relation on $\mathsf{C}+\mathsf{D}$ and $(\mathsf{C}+\mathsf{D})/\theta$ is isomorphic to $\mathsf{C}/{\phi} +\mathsf{D}/{\psi}$.
\end{corollary}

It is natural to wonder whether sums are compatible with minimum quotients. The next example illustrates that these two operations do not commute. It also demonstrates that the sum of simple games is usually not simple. 

\begin{example}
	\label{ex:compatibility of sum and quotient}
	\rm
	Let $\mathsf{D}$ be the optiongraph given in Figure~\ref{fig:compatibility of sum and quotient}. It is clear that $\mathsf{D}$ is simple, and so $\mathsf{D}/{\bowtie}+\mathsf{D}/{\bowtie}\cong \mathsf{D}+\mathsf{D}$. However, $(\mathsf{D}+\mathsf{D})/{\bowtie}$ results in the optiongraph also given in Figure~\ref{fig:compatibility of sum and quotient}. We see that $\mathsf{D}/{\bowtie}+\mathsf{D}/{\bowtie}\not\cong (\mathsf{D}+\mathsf{D})/{\bowtie}$. 
	Yet, $(\mathsf{D}+\mathsf{D})/{\bowtie}\cong(\mathsf{D}/{\bowtie+\mathsf{D}}/{\bowtie})/{\bowtie}$, which the following result generalizes for any pair of optiongraphs.
\end{example}

\begin{figure}[h]
	\hfil
	\begin{tikzpicture}[scale=.42]
		\node (1) [tiny vert] at (0,1) {};
		\node (2) [tiny vert] at (0,0) {};
		\path [d] (1) to (2);
		\node at (0,-1) {$\DD$};
	\end{tikzpicture}
	\hfil
	\begin{tikzpicture}[scale=.42]
		\node (02) [tiny vert] at (0.5,2) {};
		\node (01) [tiny vert] at (0,1) {};
		\node (11) [tiny vert] at (1,1) {};
		\node (00) [tiny vert] at (.5,0) {};
		\path [d] (02) to (01);
		\path [d] (02) to (11);
		\path [d] (01) to (00);
		\path [d] (11) to (00);
		\node at (.5,-1) {$\mathsf{D}/{\bowtie}+\mathsf{D}/{\bowtie}\cong \mathsf{D}+\mathsf{D}$};
	\end{tikzpicture}
	\hfil
	\begin{tikzpicture}[scale=.42]
		\node (02) [tiny vert] at (0,2) {};
		\node (01) [tiny vert] at (0,1) {};
		\node (00) [tiny vert] at (0,0) {};
		\path [d] (02) to (01);
		\path [d] (01) to (00);
		\node at (0,-1) {$(\mathsf{D}+\mathsf{D})/{\bowtie}$};
	\end{tikzpicture}
	\caption{
		\label{fig:compatibility of sum and quotient}
		An example showing that the sum and minimum quotient operations do not commute.
	}
\end{figure}
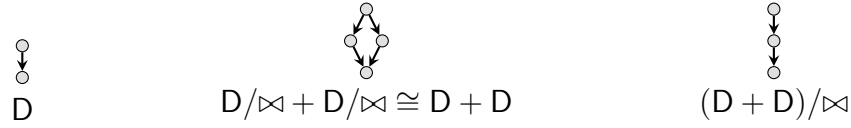

The next result is the best we can hope for in terms of compatibility between optiongraph sums and minimum quotients.

\begin{corollary}\rm 
	If $\mathsf{C}$ and $\mathsf{D}$ are optiongraphs, then $(\mathsf{C}+\mathsf{D})/{\bowtie}\cong(\mathsf{C}/{\bowtie+\mathsf{D}}/{\bowtie})/{\bowtie}$.
\end{corollary}

\begin{proof}
	Corollary~\ref{cor: quotient of a sum} shows that a quotient of $\mathsf{C}+\mathsf{D}$ is isomorphic to $\mathsf{C}/{\bowtie}+\mathsf{D}/{\bowtie}$. Therefore, $\mathsf{C}+\mathsf{D}$ and $\mathsf{C}/{\bowtie}+\mathsf{D}/{\bowtie}$ must have isomorphic minimum quotients by Corollary~\ref{cor:iso quotients, iso min quotients}.
\end{proof}

\noindent
\section*{Acknowledgment}
We thank Aaron Siegel for some valuable discussions, as well as the anonymous referees for several useful suggestions. Some of the computations were performed on Northern Arizona University’s Monsoon computing cluster, funded by Arizona’s Technology and Research Initiative Fund.

%%%%%%%%%%%%%%%%%%%%%%%%%%%%%%%%%%%%%%%%%%%%%%%%%%%%%%%%%%%%%%%%%%%%
\footnotesize
\bibliographystyle{abbrv}

\end{document}